\documentclass[12pt]{article}

\usepackage{amssymb,amsmath,amsthm}

\usepackage{hyperref}

\newenvironment{pf}[1][Proof]{\begin{proof}[\textnormal{\textbf{#1}}]}{\end{proof}}

\DeclareMathOperator{\td}{td}  

\newcommand{\N}{\ensuremath{\mathbb{N}}}
\newcommand{\Z}{\ensuremath{\mathbb{Z}}}
\newcommand{\Q}{\ensuremath{\mathbb{Q}}}

\newcommand{\C}{\ensuremath{\mathbb{C}}}

\renewcommand{\L}{\ensuremath{\mathcal{L}}} 

\renewcommand{\phi}{\varphi}

\newcommand{\into}{\hookrightarrow}

\newcommand{\nstrong}{\ensuremath{\not\kern-4pt\lhd\;}} 

\newbox\noforkbox \newdimen\forklinewidth
\forklinewidth=0.3pt \setbox0\hbox{$\textstyle\smile$}
\setbox1\hbox to \wd0{\hfil\vrule width \forklinewidth depth-2pt
  height 10pt \hfil}
\wd1=0 cm \setbox\noforkbox\hbox{\lower 2pt\box1\lower
2pt\box0\relax}
\def\unionstick{\mathop{\copy\noforkbox}\limits}

\def\nonfork_#1{\unionstick_{\textstyle #1}}

\setbox0\hbox{$\textstyle\smile$} \setbox1\hbox to \wd0{\hfil{\sl
/\/}\hfil} \setbox2\hbox to \wd0{\hfil\vrule height 10pt depth
-2pt width
                \forklinewidth\hfil}
\newbox\doesforkbox
\setbox\doesforkbox\hbox{\lower 2pt\box1 \lower
2pt\box2\lower2pt\box0\relax}
\def\nunionstick{\mathop{\copy\doesforkbox}\limits}

\def\fork_#1{\nunionstick_{\textstyle #1}}

\DeclareMathOperator{\height}{height}

\newtheorem{prop}{Proposition}
\newtheorem{cor}[prop]{Corollary}
\newtheorem{theorem}[prop]{Theorem}
\newtheorem{lemma}[prop]{Lemma}
\newtheorem{axiom}[prop]{Axiom}

\newtheorem{remark}[prop]{Remark}

\theoremstyle{definition}
\newtheorem{defn}[prop]{Definition}
\newtheorem{example}[prop]{Example}

\newcommand{\mK}{\mathcal{K}}
\newcommand{\x}{\bar{x}}
\newcommand{\ba}{\bar{a}}
\newcommand{\bb}{\bar{b}}
\newcommand{\bh}{\bar{h}}

\newcommand{\z}{\bar{z}}
\newcommand{\bu}{\bar{u}}
\newcommand{\bv}{\bar{v}}
\newcommand{\y}{\bar{y}}

\newcommand{\br}{\bar{r}}
\newcommand{\bw}{\bar{w}}

\newcommand{\mF}{\mathcal{F}}
\newcommand{\then}{\Rightarrow}
\DeclareMathOperator{\fib}{fib}

\title{Henson and Rubel's Theorem for Zilber's Pseudoexponentiation}

\author{Ahuva C. Shkop}

\begin{document}

\maketitle
\begin{abstract}
\noindent In $1984$, Henson and Rubel (\cite{H-R}) proved the following theorem: If $p(x_1,...,x_n)$ is an exponential polynomial with
coefficients in $\C$ with no zeroes in $\C$, then $p(x_1,...,x_n)= e^{g(x_1,...,x_n)}$ for some exponential polynomial
$g(x_1,...,x_n)$ over $\C$. In this paper, I will prove the analog of this theorem for Zilber's Pseudoexponentiation directly from the axioms.  Furthermore, this proof relies only on the existential closedness axiom without any reference to Schanuel's conjecture.
\end{abstract}

\section{Introduction}

In \cite{P-exp}, Zilber constructed an exponential field, Zilber's Pseudoexponentiation, of size continuum that satisfies many special properties. Schanuel's conjecture is true in this field and every definable set is countable or co-countable (quasiminimality). It is still unknown whether Pseudoexponentiation is isomorphic to complex exponentiation.

In \cite{H-R}, Henson and Rubel prove that the only exponential polynomials with no zeros are of the form $\exp(g)$ where $g$ is some exponential polynomial.  Although this seems to be a question in exponential algebra, this proof uses Nevanlinna theory.

The goal of this paper is to prove the following theorem:

\begin{theorem}\label{Thm1} Let $p(x_1,...,x_n)$ be an exponential polynomial with coefficients in Zilber's Pseudoexponentiation $\mK$. If $p \neq \exp(g(x_1,...,x_n))$ for any exponential polynomial $g(x_1,...,x_n)$, then $p$ has a root in $\mK$.

\end{theorem}

 D'Aquino, Macintyre, and Terzo have also explored this problem and offer an alternate proof of this theorem in \cite{MDT}.  We will use purely algebraic techniques and give a proof directly from the axioms.  This proof uses only basic exponential algebra and is entirely independent of Schanuel's conjecture.

\bigskip

We will begin with the following definitions.

\begin{defn}
In this paper, a \emph{(total) E-ring} is a $\Q$-algebra $R$ with
  no zero divisors, together with a homomorphism
  $\exp: \langle R,+ \rangle \to \langle R^*,\cdot \rangle$.

  A \emph{partial E-ring} is a $\Q$-algebra $R$ with no zero
  divisors, together with a $\Q$-linear subspace $A(R)$ of $R$ and a
  homomorphism $\exp: \langle A(R),+ \rangle \to \langle R^*,\cdot \rangle$. $A(R)$ is then the domain of $\exp$.

  An  \emph{E-field} is an E-ring which is a field.

  We say \emph{$S$ is a partial E-ring extension of $R$} if $R$ and $S$ are partial E-ring, $R\subset S$, and for all $r\in A(R)$,
  $\exp_S(r)=\exp_R(r)$.

\end{defn}

We now set some conventions. Let $K$ be any algebraically closed field and $\alpha \in \N$. Throughout this paper, a \emph{variety} is a (possibly reducible) Zariski closed subset of $G_\alpha (K) := K^\alpha \times (K^*)^\alpha$ or some projection of $G_\alpha (K)$. We will use the notation $\y$ for a finite tuple $y_1,...,y_m$,
and we will write $\exp(\y)$ instead of $\exp(y_1),...,\exp(y_n)$.  Similarly for a subset $S$ of an E-ring, $\exp(S)$ is
the exponential image of $S$.  We write $\td_A(\bb)$ to be the transcendence degree of the field generated by $(A,\bb)$ over the field generated by $A$.

\smallskip
To prove Theorem~\ref{Thm1}, we recall that Zilber's field, which we will call $\mK$, satisfies the following axiom:

\begin{axiom} \label{EC} If a variety $V\subseteq G_{\alpha}(\mK)$ is irreducible, rotund, and free, then there are infinitely many $\x$ such that
$(\x,\exp(\x))\in V$.
\end{axiom}
\bigskip

\noindent The definitions of a \begin{em}rotund\end{em} variety and a \begin{em}free\end{em} variety will be given later in the paper. The outline of the proof is as follows:

\begin{enumerate}
\item Given an exponential polynomial $p(\x)$, we construct a variety $V_p$ satisfying $$\forall
    \ba (\exists \bb (\ba,\bb,\exp(\ba),\exp(\bb))\in V_p \iff p(\ba)=0).$$
\item We reduce to the case where $V_p$ is irreducible and free.
\item We prove that if $p(\x)\neq \exp(g(\x))$, then $V_p$ is rotund.
\end{enumerate}

\section{Constructing $V_p$}

\noindent Recall the following construction of $K[X]^E$, the exponential polynomial ring over an E-field $K$ on the set of
indeterminates $X$: (see \cite{Lou E-rings},\cite{Angus E-rings})

\bigskip

If $R$ is a partial E-ring, we can construct $R'$, a partial E-ring extension of $R$, with the following properties:
\begin{itemize}
\item The domain of the exponential map in $R'$ is precisely $R$.
\item If for $i=1,...,n$, $y_i \notin A(R)$, then $\td_R(\exp_{R'}(\y))$ in $R'$ will be exactly the $\Q$-linear dimension of $\y$ over
    $A(R)$.
\item $R'$ is generated as a ring by $R\cup \exp(R)$.
\end{itemize}

For $K$ an E-field and $X$ a set of indeterminates, let $K[X]$ be the partial E-ring where $A(K[X])= K$. Then the exponential polynomial ring over $K$, $K[X]^E$,  is simply the union of the chain

 \[ K[X]= R_0 \into R_1 \into R_2 \into R_3 \into R_4  \into \cdots \]

where $R_{n+1} = R_n'$.

\medskip

\noindent This construction yields a natural notion of height.

\begin{defn}
For $p$ an exponential polynomial and $n \in \N$, the \emph{$\height(p)=n$} if and only if $p \in R_n$ and $p \notin
R_{n-1}$.
\end{defn}

\begin{example}
The exponential polynomial $p(x_1,x_2)= \exp(\exp(\frac{x_1}{2} + x_2^2)) + x_1^3$ in $\C[x_1,x_2]^E$  has
height $2$.

\end{example}

\noindent We now have the background necessary to begin the construction of $V_p$.

\bigskip
\bigskip

Let $K$ be an algebraically closed E-field of characteristic $0$ and $p$ an exponential polynomial with coefficients in $K$.

\begin{defn} We will call a set $T$ of exponential polynomials a \emph{decomposition of $p$} if it is a minimal set of
exponential polynomials such that:

\begin{itemize}
\item $\exists t_1,...,t_k \in T : p \in K[\x,\exp(t_1),...,\exp(t_k)]$, the subring of $K[\x]^E$ generated by $\x, \exp(t_1),...,\exp(t_k).$
\item $t_i \in T \then \exists t_1,...,t_l \in T : t_i \in K[\x,\exp(t_1),...,\exp(t_l)].$
\item There is an $L \in \Z^*$ such that $\frac{x_1}{L},...,\frac{x_n}{L}$ are in $T$.
\end{itemize}

\end{defn}

\noindent We will call elements of T \emph{T-bricks}.
\medskip

Consider the parallel between exponential polynomials and terms in the language $\L = \{+,-,\cdot, 0,1, \exp\} \cup \{c_k : k \in K\}$.  This parallel extends to subterms and T-bricks.  Considering this parallel, notice that every T-brick can be written as a polynomial in $\x$ and the exponential image of the T-bricks of lower height.  Furthermore, all decompositions are finite. To satisfy the third bullet consider the following: While there are several terms which correspond to the same polynomial, we can choose one such term and take the least common multiple of the denominators of the rational coefficients of all the elements of $\x$ which appear in the term.

\begin{example}
Consider $p(x_1, x_2)= \exp(\exp(\frac{x_1}{2}+ x_2^2)) +x_1^3$. Then $T = \{\frac{x_1}{2}, \frac{x_2}{2}, x_2^2, \exp(\frac{x_1}{2}+
x_2^2)\}$ is a decomposition of $p$. Notice that $\frac{x_1}{2}+ x_2^2$ is not in the decomposition since
$\exp(\frac{x_1}{2}+ x_2^2) = \exp(\frac{x_1}{2})\exp(x_2^2)$. We need $\frac{x_2}{2}$ in the decomposition to satisfy the third bullet. \end{example}

\medskip

\begin{defn}  We say that a decomposition $T$ is a \emph{refined decomposition} if $T$ is $\Q-$linearly independent over
$K$.
\end{defn}

\begin{lemma}  Given a decomposition $T$, we can form a refined decomposition $T'$.
\end{lemma}
\begin{proof} We induct on the size of $T$. Clearly, if the decomposition is empty, it is refined.  Suppose T is not refined, and $|T| = m$ and assume
the claim for decompositions of size less than $m$. Suppose $t\in T$ is a $\Q-$linear combination over $K$ of other $T$-bricks. That is, for all $i\in I\subset \{1,...,m\}$, $t \neq t_i$ and $$t = a + \sum_{i\in I}
\frac{a_i}{b_i}t_i$$ for some $a_i,b_i \in \Z, b_i \neq 0,a \in K $, and the least common multiple of the $b_i$ is $L$. ($L > 1$ since otherwise $T$ is not minimal and thus not a decomposition.)
Then after replacing each $t_i, i\in I$ with  $\frac{1}{L}t_i$, this set will contain a decomposition of p. $\exp(t)$ is now a polynomial in the variables $\exp(\frac{1}{L}t_i)$ So
 $\hat{T} = (T \cup \{ \frac{t_i}{L} : i\in I\}) \backslash (\{t\}\cup \{t_i: i \in I\})$ contains a smaller decomposition and by induction, we can find a refined decomposition $T'$ of $p$.

\end{proof}
\smallskip

\begin{remark}
To simplify notation, let $\x' = (\frac{x_1}{L},...,\frac{x_n}{L})$.  By making this invertible change of variables, we may and do assume $L=1$.
\end{remark}

We now set $T_0$ to be a refined decomposition of $p$, and $\alpha = |T_0|$.  Furthermore, we order the $T_0$-bricks in
order of height, i.e., $\height(t_i) \leq \height(t_j)$ for $i \leq j $. For convenience, we let the first $n$
elements of $T_0$ be $x_1,...,x_n$.  So $T_0= \{\x\} \cup \{t_i: n+1 \leq i \leq \alpha \}$.

\medskip

We now name the polynomials which witness $T_0$ being a decomposition.

\noindent For each $ n+1< i < \alpha $, let $p_i \in K[\x,y_1,...,y_{i-1}]$ be such that  $$p_i(\x,\exp(\x),\exp(t_{n+1}),...,\exp(t_{i-1}))=t_i$$
Let $p^* \in K[\x,\y]$ be such that $p^*(\x,\exp(\x,t_{n+1},...,t_\alpha))=p$.

\noindent Let $V_p \subseteq G_\alpha = K^\alpha \times (K^*)^\alpha$ be the variety given as follows:

$$(x_1,...,x_n,w_{n+1},...,w_\alpha,y_1,...,y_\alpha)=(\x,\bw,\y)\in V_p$$ $$ \iff (\bigwedge_{i=n+1}^\alpha w_i = p_i(\x,\y)) \wedge p^*(\x,\y)= 0$$

Please note the indexing. We will maintain this indexing for coordinates of points in the variety as well.
\bigskip

\begin{prop} For any $\ba \in K$
$$\exists \bb ( (\ba,\bb,\exp(\ba),\exp(\bb)) \in V_p ) \iff p(\ba)=0$$
\end{prop}

\begin{proof}

$$\exists \bb ( (\ba,\bb,\exp(\ba),\exp(\bb))\in V_p)\iff$$ $$ \exists \bb ( \forall b_i \in \bb (b_i=
p_i(\ba,\exp(\x'),\exp(b_{n+1},...,b_{i-1}))) \wedge p^*(\ba,\exp(\ba),\exp(\bb))=0)$$ $$\iff \exists \bb ( \forall b_i \in \bb
(b_i = t_i(\ba)) \wedge p^*(\ba,\exp(\ba), \exp(\bb)) =0 )$$

Since $p^*(\ba,\exp(\ba),\exp(t_{n+1}(\ba)),...,\exp(t_\alpha(\ba)))= p(\ba)$, this is if and only if $p(\ba) =0 $.

\end{proof}

\noindent This concludes the general construction of $V_p$.

We now fix an algebraically closed E-field $K$ of characteristic $0$ whose exponential map is surjective with infinite kernel, and an exponential polynomial $p(\x)=p(x_1,...,x_n)$ with coefficients in $K$ of height at least 1.  Since the only polynomials with no zeros are constant and non-zero, theorem \ref{Thm1} is clearly true for polynomials.

Notice that if $p=\exp(g(\x))$, then $p^* = \prod y_i^{a_i}$ and $ \prod y_i^{a_i} =0$ is one of the defining equations of $V_p$ and $V_p$ is empty. So for the remainder of the paper, we assume $p \neq \exp(g)$ for any exponential
polynomial $g$. Furthermore, we have set $T_0$ to be the refined decomposition which gave us $V_p$. To prove Theorem \ref{Thm1}, we now need to show that we can reduce to the case where $V_p$ is irreducible and free. This is a necessary step to use axiom \ref{EC} for Zilber's field.

\section{Irreducibility and Freeness}

We now reduce to the case where $V_p$ is irreducible and free.  These reductions involve two inductive procedures on $p$.  One decreases the height and the other does not increase the height so this process will terminate.

\begin{defn} An exponential polynomial $p$ is \begin{em}irreducible with respect to a decomposition T\end{em}, if there are
no nonconstant exponential polynomials $q_1,q_2$ such that

\begin{itemize}
\item T contains decompositions for $q_1,q_2$
\item $p=q_1q_2$
\end{itemize}
\end{defn}

When T is a refined decompostion of $p$, this is equivalent to demanding that $p$ be irreducible as a polynomial in the polynomial ring $K[\x,\exp(T)]$ . Note that $(\x,\exp(T))$ is algebraically independent over $K$ in the exponential polynomial ring by construction so $K[\x,\exp(T)]$ is isomorphic to a polynomial ring.  As this ring is a unique factorization domain, $p$ can be written as a product of nontrivial irreducibles, say $\{q_j\}$.  If each factor $q_j$ is equal to $\exp{g_j}$ for some exponential polynomial
$g_j$, then $p=\exp(\sum g_j)$. So $p\neq \exp(g)$ for any exponential polynomial $g$ implies that there is an irreducible factor of $p$, say $q_j$ such that $q_j \neq \exp(\hat{g})$ for any exponential polynomial $\hat{g}$. Furthermore, if $g_j$ has a root, then $p$ has a root, $T_0$ contains a refined decomposition $T_1$ of $q_j$, and $q_j$ is clearly irreducible with respect to $T_1$. So to prove Theorem 1, we can assume that $p$ is
irreducible with respect to $T_0$.  It is also clear that if $p$ is irreducible with respect to $T_0$, $p^*$ is an
irreducible polynomial.

\begin{lemma}\label{irreducible lemma}
If $p$ is irreducible with respect to $T_0$, then $V_p$ is irreducible.
\end{lemma}

\begin{pf}
Consider the projection $\phi: V_p \to K^n \times (K^*)^{\alpha}$ where $$\phi(\x,\bw,\y) = (\x,\y)$$  This map is injective since
every element of $\bw$ is determined by $\x,\y$.  The inverse is given by the polynomial map $\phi^{-1}(\x,\y) =
(\x,p_{n+1}(\x,\y),...,p_{\alpha}(\x,\y),\y)$. Thus $V_p$ is isomorphic to the image of $\phi$. The image is defined by $p^*(\x,\y)=0$.  This is a hypersurface given by an irreducible polynomial and is thus irreducible. Since it is
isomorphic to an irreducible variety, $V_p$ is irreducible.

\end{pf}

\begin{defn}  A variety $V \subseteq G_{\alpha}(K)$ is \emph{free} if we cannot find $m_1,...,m_\alpha \in \Z$ and $b \in K$
such that $V$ is contained in either the variety $$\{(\bu,\bv): \prod_{i=1}^\alpha v_i^{m_i}=b\}$$ or
$$\{(\bu,\bv):\sum_{i=1}^\alpha m_iu_i=b\}.$$
\end{defn}

\begin{lemma} If $V_p$ is not free, then $p= \exp(g)-k$ for some exponential polynomial $g$ and some $k \in K$.
\end{lemma}
\begin{pf}
Suppose $V_p$ is not free. Since we demanded that the $T_0$ bricks be $\Q$ - linearly independent over $K$, we cannot find
$m_1,...,m_\alpha \in \Z$ and $b \in K$ such that $V_p$ is contained in the variety $$\{(\x,\bw,\y):\sum_{i=1}^n m_ix_i +
\sum_{i=1}^{\alpha-n}m_{i+n}w_i =b\}.$$  Suppose $V_p$ is contained in the variety $W:=\{(\x,\bw,\y): \prod_{i=1}^\alpha
y_i^{m_i}=b\}$ for some $b\in K$.  Then, if $\phi$ is defined as in Lemma \ref{irreducible lemma}, consider $V_p^*= \phi(V_p)$ and $W=\{(\y): \prod_{i=1}^\alpha y_i^{m_i}=b\}$. As
stated above, $V_p^*$ is a hypersurface given by $p^*(\x,\y)=0$.  If $V_p \subseteq W$ then $V_{p}^*\subseteq W^*:= \phi(W)=
\{(\x,\y): \prod_{i=1}^\alpha y_i^{m_i}=b\}$.

\bigskip

If $W^*$ is reducible, then there is $m=gcd(\{m_1,...,m_\alpha\})\neq 1$ and $$\prod_{i=1}^\alpha y_i^{m_i}-b = (\prod_{i=1}^\alpha y_i^{\frac{m_i}{m}})^m - b$$ \noindent  Since $K$ is an algebraically closed field, it contains all the
roots of $b$, and thus the irreducible factors of $\prod_{i=1}^\alpha y_i^{m_i}-b$ are all of the form
$\prod_{i=1}^\alpha y_i^\frac{m_i}{m}-b_j$ so it suffices to prove the claim for $W^*$ irreducible.

\medskip

If $W^*$ is irreducible, then, $p^*$ must divide $\prod_{i=1}^\alpha y_i^{m_i} - b$.  Since $W^*$ is irreducible,
and $\prod_{i=1}^\alpha y_i^{m_i} - b$ is irreducible, we know that $p^* = k'(\prod_{i=1}^\alpha y_i^{m_i}- b)$ for some $k'\in K$.  From the construction of $p^*$, there is an exponential polynomial $g$ ($g = \sum m_it_i+log(k')$) such that    $p= \exp(g)-bk'$.

\end{pf}

\begin{lemma} If $V_p$ is not free, we can find an exponential polynomial $p'$ so that the $\height(p') < \height(p)$ and if $p'(\x)=0$ then $p(\x) = 0$.

\end{lemma}
\begin{proof}
If $V_p$ is not free, then  $p= \exp(g) -b$ for some $b \in K^*$. (We are assuming that $p\neq \exp(g)$).  We can find $
\log(b)\in  K^*, \exp(\log(b))=b$. (Note: $\exp$ is not injective, so there are non-zero elements of the kernel allowing a
non-zero choice for $\log(1)$.) Then we can find zeroes of $p' = g - \log(b)$ which is now of lower height than $p$.
\end{proof}

\begin{cor}\label{free or poly} If $V_p$ is not free, we can always find a $\hat{p}$ such that $\hat{p}(\x) = 0 \Rightarrow p(\x)=0$, and
either $V_{\hat{p}}$ is free or $\hat{p}$ is a polynomial.
\end{cor}

\begin{proof}
 By the previous lemma, if $V_p$ is not free, we can find an exponential polynomial of lower height, $p'$ such that
 $p'(\x)=0 \Rightarrow p(\x) = 0$.  Iteration of this process will yield the desired result.

\end{proof}

Once again, since the only polynomials with no zeroes are the constant non-zero polynomial, Theorem 1 is clearly true for polynomials
in any field with a surjective exponential map. Furthermore, if $p$ is a non-constant exponential polynomial, $\hat{p}$
will not be constant. Thus, we have now reduced to the cases where either $\hat{p}$ is free or we can find solutions to $p$ by solving polynomials.  So we need only prove the theorem for exponential polynomials $p$ where  $\height(p) \geq 1$,
$T_0$ is a refined decomposition, and $V_p$ is irreducible and free.  All that remains is to show that under these
circumstances, $V_p$ is rotund.

\section{Rotundity of $V_p$}

\begin{defn} Let $C=(c_{i,j})$ be an $r \times \alpha$ matrix of integers and let

\noindent $[C]:G_\alpha(K) \to G_r(K)$ be the
function $[C](\z,\y)= (u_1,...,u_r,v_1,...,v_r)$ where $$u_i = \sum_{j=1}^\alpha c_{i,j}z_j  \text{  and  }  v_i =
\prod_{j=1}^\alpha y_j^{c_{i,j}}.$$   An irreducible variety $V \subseteq G_\alpha(K)$ is \emph{rotund} (normal in
\cite{Markerremarks}, ex-normal in \cite{P-exp}) if $\dim([C](V))\geq r$ for any $r \times \alpha$ matrix of integers C of rank $r$ where $1\leq r\leq \alpha$.

\end{defn}

\begin{lemma}
For any exponential polynomial $p(x_1,...,x_n)$, if $V_p$ is an irreducible and free variety defined via a refined
decomposition, $V_p$ is rotund.
\end{lemma}

\begin{proof}

Let $C=(c_{i,j})$ be an $r \times \alpha$ matrix of integers of rank $r$. To prove this lemma, we will use the fiber
dimension theorem (see \cite{Shafarevich}) which tells us that $\dim ([C] (V_p)) = \dim(V_p)- \dim(\fib(d))$ where $\fib(d) = [C]^{-1}(d)$ and $d$ is a generic point in $[C](V_p)$.

By simply counting the number of equations by which $V_p$ is defined, we know that $\dim(V_p) \geq 2\alpha - (\alpha -
n) -1 = \alpha +n -1$.  Let $\ba\bb = (a_1,...,a_r,b_1,...,b_r)$ be a generic point in $[C](V_p)$.  By the fiber dimension
theorem, we know that $\dim(V_p) - \dim(\fib(\ba\bb)) = \dim ([C] (V_p)) $.  So it suffices to show that
$\dim(\fib(\ba\bb)) < \alpha + n - 1 - r$.

\medskip
\medskip
Consider the equations that define the fiber, $\mF$:

\noindent We have for $i=n+1,...,\alpha$, $$p_i(\x,y_1,...,y_{i-1}) = w_i$$ and we have $$p^*(\x,\y)=0$$ and for each $
j=1,...,r$ we have \begin{equation} c_{1,j}x_1+...+ c_{n,j}x_n+ c_{n+1,j}w_{n+1}+,...,c_{\alpha,j}w_{\alpha}=a_j\end{equation} and
\begin{equation}\label{mult} y_1^{c_{1,j}}\cdots y_\alpha^{c_{\alpha,j}}=b_j\end{equation}

\medskip

Consider the projection $\phi: G_\alpha(K) \to K^n \times (K^*)^{\alpha}$ defined by $\phi(\x,\bw,\y) = (\x,\y)$. Since $\phi \upharpoonleft V_p$ is an isomorphism, we know $\dim(\mF) = \dim(\phi(\mF))$. Let  $V_2 \subset \phi(G_\alpha(K))$ be the variety given by the multiplicative equations in $(2)$.  Since $V_2$ is defined by $r$ independent multiplicative equations, $\dim(V_2) = \alpha + n - r$.  If $F_0$ is the field of definition of $V_p$, it suffices to show that $\dim(\phi(\mF)) \lneq \dim(V_2)$, i.e., if $(\br,\bh)$ is a generic point of $V_2$ over $F_0(\ba\bb)$, $(\br,\bh) \notin \phi(\mF)$.

Let $(\br,\bh)$ be a generic point of $V_2$ and let $\beta$ be the maximum such that $(\br,h_1,...,h_\beta)$ is algebraically independent. So $(h_1,..., h_{\beta +1})$ are algebraically dependent. Therefore, there is some tuple  of integers $\overline{\lambda}$ and
 integers $d_j = \sum \lambda_j c_{i,j}$ and multiplicative equation $$ \prod_{j=1}^{\beta + 1} y_j^{d_j}= \prod_{j=1}^{\beta + 1}b_j^{\lambda_j}$$ which is satisfied in $V_2$ so it must also be satisfied by $(h_1,..., h_{\beta +1})$.  Now consider the linear equation which must also be satisfied in $\fib(\ba\bb)$: $$d_1x_1+...+d_nx_n +
 d_{n+1}p_{n+1}(\x)+...+d_{\beta +1}p_{\beta+1}(\x,y_1,...,y_\beta)= \sum_{j=1}^{\beta + 1} \lambda_j a_j$$ The left-hand side of this equation is a nonconstant polynomial in the variables $\x,y_1,...,y_\beta$ because the $T_0$-bricks are $\Q$-linearly independent over $K$.  However, this equation must be satisfied in $\fib(\ba\bb)$.  Thus $(\br,\bh) \notin \fib(\ba\bb)$.

\end{proof}

\noindent We are now done.  This proof yields the following corollary:

\begin{cor}
Suppose $p(x) \in \mK[x]^E$ and $p(x)=0$ has exactly $m>0$ many solutions for some $m\in \N$.  Then there are $a_1,...,a_m \in \mK, n_1,...,n_m \in \N$ and an exponential polynomial $g$ such that $$p(x) = (x-a_1)^{n_1}\cdots (x-a_m)^{n_m}\exp(g).$$

\end{cor}

\begin{proof}
Since $p$ has a zero, $p \neq \exp(g)$ for any exponential polynomial $g$.  Let $V_p$ be a variety given by a refined decomposition of $p$. We've shown that if $V_p$ is irreducible and free, then $V_p$ is rotund and $p$ has infinitely many solutions. Furthermore, if $p$ has only finitely many solutions, every factor of $p^*$ can lead to only finitely many solutions.  So it suffices to consider $V_p$ irreducible but not free.   Notice that if $p = \exp(g) - k$ for some $k\in \mK$, there are infinitely many choices for $\log(k)$ and thus infinitely many zeros.  So $p$ must be a polynomial.  (We excluded this case on page $3$.) The only irreducible polynomials with finitely many solutions are lines.
\end{proof}

\end{document}